\documentclass[11pt,oneside]{article}
\usepackage[T1]{fontenc}
\usepackage[utf8]{inputenc}
\usepackage{amsmath,amssymb,amsthm,verbatim}
\usepackage[hyperindex,breaklinks]{hyperref}
\newtheorem{theorem}{Theorem}[section]
\newtheorem{lemma}[theorem]{Lemma}
\newtheorem{prop}[theorem]{Proposition}
\newtheorem{question}[theorem]{Question}
\newtheorem{conj}[theorem]{Conjecture}
\PassOptionsToPackage{linktocpage}{hyperref}

\newcommand{\GM}{G[\mathcal{M}]}
\DeclareMathOperator{\sign}{sgn}
\newcommand{\f}{\mathcal{F}}

\title{A collection of open problems in celebration of Imre Leader's 60th birthday}
\date{Monday 30th October, 2023}

\usepackage[a4paper]{geometry}

\begin{document}

\maketitle


\section*{Dedication}

One of the great pleasures of working with Imre Leader is to experience his infectious delight on encountering a compelling combinatorial problem.

This collection of open problems in combinatorics has been put together by a subset of his former PhD students and students-of-students for the occasion of his 60th birthday. 

All of the contributors have been influenced (directly or indirectly) by Imre: his personality, enthusiasm and his approach to mathematics. The problems included cover many of the areas of combinatorial mathematics that Imre is most associated with: including extremal problems on graphs, set systems and permutations, and Ramsey theory.

This is a personal selection of problems which we find intriguing and deserving of being better known. It is not intended to be systematic, or to consist of the most significant or difficult questions in any area. Rather, our main aim is to celebrate Imre and his mathematics and to hope that these problems will make him smile. We also hope this collection will be a useful resource for researchers in combinatorics and will stimulate some enjoyable collaborations and beautiful mathematics.

\hfill David Ellis

\hfill Robert Johnson

\hfill (Editors and contributors);
\vspace{0.25cm}

\hfill Rahil Baber

\hfill Natalie Behague

\hfill Asier Calbet

\hfill Joshua Erde

\hfill Ron Gray

\hfill Maria-Romina Ivan

\hfill Barnab\'as Janzer

\hfill Luka Mili\'cevi\'c

\hfill John Talbot

\hfill Ta Sheng Tan 

\hfill Belinda Wickes

\hfill (Contributors).

\pagebreak

  \tableofcontents

\pagebreak

\section{Tur\'an problems in the hypercube (Rahil Baber and John Talbot).}

Let $\mathcal{Q}^n$ be the $n$-dimensional hypercube: the graph with vertex set $\{0,1\}^n$ and edges between any two vertices that differ in exactly one bit.

A \emph{$k$-subcube} is a sequence $s\in  \{0,1,*\}^n$ with exactly $k$ coordinates equal to $*$. So a $0$-subcube is a vertex of $\mathcal{Q}^n$ and a $1$-subcube is an edge. We denote the set of all $k$-subcubes of $\mathcal{Q}^n$ by $\mathcal{Q}^n_k$.
Given $s\in \mathcal{Q}^n_k$, we say that $i \in [n]$ is a \emph{flip-bit} of $s$ if $s(i)=*$ (where $s(i)$ denotes the $i^\textrm{th}$ coordinate of $s$).   A $d$-subcube $s_d\in \mathcal{Q}^n_d$  \emph{contains} the $k$-subcube $s_k\in \mathcal{Q}^n_k$ if  $s_k(i)=s_d(i)$ for every non-flip-bit $i$ of $s_d$ (in particular every flip-bit of $s_k$ is also a flip-bit of $s_d$ and $k\le d$). For example, $(0,*,*,*)$ is a 3-subcube of $\mathcal{Q}^4$ containing the following six $2$-subcubes: $(0,0,*,*),(0,1,*,*),(0,*,0,*),(0,*,1,*),(0,*,*,0)$ and $(0,*,*,1)$. 

We say that a set $S \subseteq \mathcal{Q}^n_k$ of $k$-subcubes is $\mathcal{Q}^n_d$-\emph{free} if for every  $d$-subcube $s_d\in \mathcal{Q}^n_d$ there is a $k$-subcube contained in $s_d$ that is not in $S$. The natural Tur\'an problem in the hypercube is: given $0 \le k \le d \le n$ what is the maximum size of a set of $\mathcal{Q}^n_d$-\emph{free} $k$-subcubes? We will denote this quantity by $\textrm{ex}_k(n,d)$.

Standard averaging arguments imply that for $0 \le k \le d$ the Tur\'an density
\[
\pi_k(d) =\lim_{n \to \infty} \frac{\textrm{ex}_k(n,d)}{|\mathcal{Q}_k^n|},\] exists.
Trivially $\pi_d(d) = 0$ and $\pi_0(1)=1/2$. The first non-trivial result in this area is due to Kostochka \cite{K} who gave an exact construction proving that $\pi_0(2)=2/3$. An old conjecture of Erd\H os  \cite{EC} implies that $\pi_1(2)=1/2$. The lower bound is easy while the current best upper bound is 0.60318 due to Baber \cite{B}.

We conjecture that for $d=3$ the non-trivial Tur\'an densities all agree.
\begin{conj}
  For $0 \le k \le 2$, $\pi_k(3)=3/4$.
\end{conj}
To motivate this conjecture we describe three constructions which imply that $\pi_k(3)\ge 3/4$ for $k=0,1,2$.

For $k=0$ take the set of all vertices of weight (i.e.~sum of bits) not divisible by $4$, this is clearly $\mathcal{Q}_3^n$-free and implies that $\pi_0(3)\ge 3/4$.  Baber \cite{B} proved the upper bound $0.76900$.

For any $1$-subcube $s\in \mathcal{Q}_1^n$ let $w_0(s)$ and $w_1(s)$ denote the sum of bits to the left and right respectively of the single flip-bit. We claim that \[
C_1^3 = \{ s \in \mathcal{Q}_1^n: w_0(s), w_1(s) \textrm{ not both even}\},\] is $\mathcal{Q}_3^n$-free, and hence $\pi_1(3)\ge 3/4$. To see this let $t$ be any $3$-subcube in $ \mathcal{Q}_3^n$  and consider the $1$-subcubes contained in $t$ whose flip-bit is the second flip-bit of $t$. 
Choosing appropriate $0/1$ values for the two other flip-bits of $t$, we can find a $1$-subcube $s$ in $t$ with $w_0(s)$ and $w_1(s)$ both even, hence $s \not\in C_1^3$.

For $k=2$, take a balanced bipartition of $[n]$ and let $D_2^3$ contain all $2$-subcubes except those with both flip-bits in the same part whose sum of non-flip-bits is even. We claim that $D_2^3$ is $\mathcal{Q}_3^n$-free and hence $\pi_2(3)\ge 3/4$. To see this consider any $3$-subcube  $t \in \mathcal{Q}_3^n$. Since $t$ has three flip-bits, at least two belong to the same part of the balanced bipartition. Choosing the 0/1 value of the third flip-bit appropriately we can find a  $2$-subcube contained in $t$ that is not in $D_2^3$. 
\medskip

Both constructions can be generalised to any choice of $k$ and $d$, we will focus on the case of $k=2$. In much the same way as Turan's theorem extends Mantel's theorem, we can generalise $D_2^3$ to show that $\pi_2(d+1)\ge 1- 1/d^2$. Take a balanced partition of $[n]$ into $d$ parts and let $D_{2}^{d+1}$ be the set of all $2$-subcubes except those with both flip-bits in the same part whose sum of non-flip-bits is divisible by $d$.

The generalisation of $C_1^3$ is a little more complicated: we restrict our description to the case $3m+2$. For any $2$-subcube $s$, the two flip-bits of $s$ partition the non-flip-bits into three parts. We denote the sums of the bits in these parts by $w_0(s),w_1(s)$ and $w_2(s)$. Let $C_{2}^{3m+2}$ consist of all $2$-subcubes $s$ for which not all of $w_0(s),w_1(s),w_2(s)$ are divisible by $m+1$. To see that $C_{2}^{3m+2}$ is $\mathcal{Q}_{3m+2}^n$-free let $t$ be any $(3m+2)$-subcube and consider the $2$-subcubes in $t$ which use the $(m+1)^\textrm{th}$ and $(2m+2)^\textrm{th}$ flip-bits of $t$. Choosing appropriate 0/1 values for the other $3m$ flip-bits of $t$ we can find a $2$-subcube missing from $C_2^{3m+2}$, hence $\pi_2(3m+2) \ge 1- 1/(m+1)^3$.

The second construction is larger for $m\ge 7$ but it would be interesting to know if there are other constructions that beat both of these. 

\begin{question} Is $\pi_2(3m+2) = 1-1/(m+1)^3$ for $m \ge 7$?
\end{question}

\pagestyle{empty}

\pagebreak

\section{Connecting factors of the hypercube (Natalie Behague).}

Let $Q_d$ be the $d$-dimensional hypercube.

\begin{question}[Behague, \cite{Behague}] For each $d \in \mathbb{N}$ with $d \geq 3$, let $r=r(d)$ be the minimal positive integer such that there exists a 1-factorization $\mathcal{M}$ of $Q_d$ where the union of any $r$ distinct 1-factors is connected. What is the value of $r(d)$?
\end{question}

Note that to have $r(d)=2$ would mean there is a 1-factorization of $Q_d$ in which the union of any two of the 1-factors forms a Hamilton cycle (compare Kotzig's perfect 1-factorization conjecture for complete graphs \cite{Kotzig}). However, a cute argument shows that $r(d)$ is greater than two, for all $d \geq 3$. According to Bryant, Maenhaut and Wanless \cite{Bryant} this has been reproved many times, including by Laufer in 1980 \cite{Laufer}.
\begin{lemma}[Laufer, \cite{Laufer}] \label{thm:Bipartite}
	Let $H$ be a bipartite graph on two vertex classes each of size $n$, where $n$ is even. 
	Let $\mathcal{M}$ be a partition of $H$ into perfect matchings. Let $\GM$ be the graph on vertex set $\mathcal{M}$ where $M_i \sim M_j \in E(\GM)$ if $M_i \cup M_j$ is connected. Then $\GM$ must be bipartite.
\end{lemma}

\begin{proof}
	Let $X$ and $Y$ be the vertex classes of $H$.
	A perfect matching $M$ naturally induces a function ${M: X \rightarrow Y}$, where $(x,M(x))$ is an edge of $M$.
	
	For two perfect matchings $M_i$ and $M_j$, let $\pi_{j,i}$ be the permutation $M_j^{-1}M_i$ on $X$.
	Note that $\pi_{i,i} = id$,  $\pi_{i,j} = \pi_{j,i}^{-1}$ and 
	$\pi_{k,j}\pi_{j,i} = \pi_{k,i}$.
	Note further that if $M_iM_j$ is an edge of $\GM$ then $M_i \cup M_j$ is a Hamilton cycle and so $\pi_{j,i}$ is a cycle of length $n$ on $X$. 
	
	Suppose for a contradiction that $\GM$ contains a odd cycle and let $M_{i_1}$, $M_{i_2}, \ldots$, $M_{i_k}$, $M_{i_1}$ be such a cycle. 
	The permutations $\pi_{i_2,i_1}$, $\pi_{i_3,i_2}, \ldots$, $\pi_{i_k,i_{k-1}}$, $\pi_{i_1,i_k}$ are all cycles of length $n$. Since $n$ is even, all of these are odd permutations. Now,
	\begin{align*} 
		1 = \sign(\pi_{i_1,i_1}) &= \sign(\pi_{i_1,i_k}\pi_{i_k,i_{k-1}} \pi_{i_{k-1},i_{k-2}} \ldots \pi_{i_3,i_2} \pi_{i_2,i_1}) \\
		&= \sign(\pi_{i_1,i_k})\sign(\pi_{i_k,i_{k-1}})\ldots \sign(\pi_{i_3,i_2})\sign(\pi_{i_2,i_1})\\
		&= (-1)^k = -1
	\end{align*}
	We have a contradiction, hence $\GM$ contains no odd cycles.
\end{proof}

A 1-factorization given in \cite{Behague} for $d \ne 6$ has the property that the union of any $\left( \left\lceil \frac{d}{2} \right\rceil + 1 \right)$ 1-factors is connected, hence $r(d) \le \left\lceil \frac{d}{2} \right\rceil + 1$ for $d \ne 6$. It seems entirely possible that $r$ is constant, and it could be even as small as 3.
\pagestyle{empty}


\pagebreak

\section{A conjectured new version of the Two Families Theorem (Asier Calbet).}

The `Two Families Theorem' of Bollob\'as is a celebrated result in extremal set theory. There are several different versions of this theorem, the simplest one being as follows. 

\begin{theorem}[Bollob\'as' `Two Families Theorem', 1965]\label{thm:original}

Let $a,b \geq 0$ be integers and $(A_i,B_i)_{i \in I}$ be a collection of pairs of finite sets, indexed by a finite set $I$, with the following properties.

\begin{enumerate}
    \item $|A_i| = a$, $|B_i| = b$ and $A_i \cap B_i = \emptyset $ for all $i \in I$.
    \item $A_i \cap B_j \neq \emptyset $ for all $i \neq j \in I$.
\end{enumerate}

Then $|I| \leq \binom{a+b}{a}$. 

\end{theorem}

One can see that Theorem \ref{thm:original} is tight by taking $(A_i,B_i)_{i \in I}$ to be the collection of all partitions of a set of size $a+b$ into a subset $A_i$ of size $a$ and a subset $B_i$ of size $b$. Moreover, this is the unique way of achieving equality. There are numerous extensions and variations of this theorem and it has many applications.
\newline

In \cite{calbet}, I prove a new version of the Two Families Theorem, which I then use to prove bounds on the minimum number of edges in a $K_r$-saturated graph with given minimum degree. I believe that to prove tight bounds for this saturation problem a further generalisation of the Two Families Theorem is needed. The simplest special case of this generalisation that I cannot prove is as follows.

\begin{conj}\label{conj}

Let $b \geq a \geq 2$ be integers and $(A_i,B_i)_{i \in I}$ be a collection of pairs of finite sets, indexed by a finite set $I$, with the following properties.

\begin{enumerate}
    \item $|A_i| = a$, $|B_i| = b$ and $|A_i \cap B_i| = 2$ for all $i \in I$.
    \item $A_i \cap B_j \not \subseteq A_k \cap B_k$ for all $i,j,k \in I$ with $i \neq j$.
\end{enumerate}

Then $|I| \leq \sum_{i=2}^a 2^{i-2} \binom{a+b-2i}{a-i}$. 

\end{conj}

The following example shows that, if true, Conjecture \ref{conj} is tight. For each integer $2 \leq c \leq a$, let $(A_i)_{i \in I_c}$ be the collection of all sets
\begin{align*} \{A:\ \{2c-3,2c-2\} \subseteq A \subseteq \{1,2,3,\cdots,a+b-2\},\ |A|=a,\\
|A \cap \{2d-3,2d-2\}|=1\ \text{for all } 2 \leq d < c\}.\end{align*}
For each $i \in I_c$, let
$$B_i=\left(\{1,2,3,\ldots,a+b-2\} \setminus A_i\right) \cup \{2c-3,2c-2\}.$$
Let $I=\bigcup_{2 \leq c \leq a} I_c$. Then $(A_i,B_i)_{i \in I}$ satisfies the conditions in Conjecture \ref{conj} and $|I|=\sum_{i=2}^a 2^{i-2} \binom{a+b-2i}{a-i}$. \newline

I can prove Conjecture \ref{conj} under the additional assumption that the $A_i$ and $B_i$ are all subsets of a ground set of size $a+b-2$ (the smallest possible size of a set containing sets of size $a$ and $b$ whose intersection has size 2). Note that taking $k=i$ in condition 2, we obtain that $(A_i \setminus B_i, B_i)_{i \in I}$ satisfies the conditions in Theorem \ref{thm:original} with $a$ replaced by $a-2$. Hence $|I| \leq \binom{a+b-2}{a-2}$. When $a=3$, this bound matches the one in Conjecture \ref{conj}, but there are many different ways of achieving equality. \newline

See \cite{calbet} for more details. In particular, Problem 2 in \cite{calbet} and Definition 11 in \cite{calbet} give the full generalisation of the Two Families Theorem needed for the saturation application. See also Section 4 in \cite{scott} for a discussion of several different versions of the Two Families Theorem and their applications.

\pagestyle{empty}

\pagebreak

\section{Some problems on intersecting families of graphs (David Ellis).}

If $H$ is a fixed (unlabelled) graph, and $\mathcal{F}$ is a family of graphs on the vertex-set $[n]:=\{1,2,\ldots,n\}$, we say $\mathcal{F}$ is {\em $H$-intersecting} if $G_1 \cap G_2$ contains a copy of $H$, for any $G_1,G_2 \in \mathcal{F}$. The study of $H$-intersecting families of graphs was initiated by Simonovits and S\'os in 1976. They posed the natural problem of determining (or bounding) the function
$$g_H(n):=\max\{|\mathcal{F}|:\ \mathcal{F} \text{ is an }H\text{-intersecting family of graphs on }[n]\},$$
for various graphs $H$. It is trivial to see that for any graph $H$ with at least one edge, we have $g_H(n) \leq 2^{{n \choose 2}-1}$, since an $H$-intersecting family of graphs on $[n]$ cannot contain both a graph and its complement. However, improving on this `trivial' upper-bound turned out to be non-trivial. For the case where $H$ is $K_3$, the triangle, Simonovits and S\'os conjectured that $g_{K_3}(n) = 2^{{n \choose 2}-3}$ for all $n \in \mathbb{N}$, i.e.\ that a triangle-intersecting family of graphs on $[n]$ with the maximal size can be obtained by fixing a single triangle and taking all graphs on $[n]$ containing that triangle. The first progress on this was made by Chung, Frankl, Graham and Shearer \cite{cfgs} in 1986; they proved, using Shearer's entropy lemma, that $g_{K_3}(n) \leq 2^{{n \choose 2}-2}$ for all $n \in \mathbb{N}$, which is halfway between the `trivial' upper bound and the sharp upper bound conjectured by Simonovits and S\'os. The conjecture of Simonovits and S\'os was proved in 2011 by the author, Filmus and Friedgut \cite{eff}, using a discrete Fourier-analytic argument. In \cite{eff}, we conjecture an analogue of the Simonovits-S\'os conjecture for the clique $K_t$, for all $t \geq 4$.
\begin{conj}
\label{conj:eff}
    For all $t,n \in \mathbb{N}$ with $4 \leq t \leq n$,
    $$g_{K_t}(n) = 2^{{n \choose 2}-{t \choose 2}}.$$
\end{conj}
This conjecture says that, for any $t \geq 4$, to obtain a largest possible $K_t$-intersecting family of graphs on $[n]$, one cannot do better than to take the `trivial' construction of fixing a single copy of $K_t$ and taking all graphs on $[n]$ containing that $K_t$.

Using an adaptation of the methods in \cite{eff}, together with some ingenious computational work, Berger and Zhao \cite{bz} proved Conjecture \ref{conj:eff} for $t=4$. However, the conjecture remains open for every value of $t \geq 5$. We have not ruled out that a Fourier-theoretic approach (similar to that in \cite{bz} and \cite{eff}) might work, but it seems very hard to make it work.

It is easy to adapt the above argument of Chung, Frankl, Graham and Shearer (in the triangle-intersecting case) to prove the following.
\begin{prop}
Let $t \geq 3$. If $\f$ is a family of graphs on the vertex-set $[n]$ such that for all $G,H \in \f$, $G \cap H$ has chromatic number at least $t$, then $|\f| \leq 2^{{n \choose 2}-t+1}$.
\end{prop}
We conjecture the following strengthening of Conjecture \ref{conj:eff}.
\begin{conj}
Let $t \geq 3$. If $\f$ is a family of graphs on the vertex-set $[n]$ such that for all $G,H \in \f$, $G \cap H$ has chromatic number at least $t$, then $|\f| \leq 2^{{n \choose 2}-{t \choose 2}}$.
\end{conj}
Somewhat embarrassingly, even the following weaker conjecture remains open.
\begin{conj}
There exists an absolute constant $c>0$ such that the following holds. Let $t \geq 3$. If $\f \subset \mathcal{P}({[n] \choose 2})$ is a family of labelled graphs on the vertex-set $[n]$ such that for all $G,H \in \f$, $G \cap H$ has chromatic number at least $t$, then $|\f| \leq 2^{{n \choose 2}-ct^2}$.
\end{conj}

In \cite{cfgs}, Chung, Frankl, Graham and Shearer also conjectured that, for the path $P_3$ with three edges,
$$g_{P_3}(n) = 2^{{n \choose 2}-3}\quad \forall n \geq 4,$$
i.e., that the `trivial' construction of fixing a single path of length three and taking all graphs on $[n]$ containing that path, is also optimal for $P_3$-intersecting families. This was disproved in 2006 by Christofides \cite{christofides}, who gave an elegant construction in the case $n=6$ disproving it (taking a product-construction, one sees from this that the conjecture is false for all $n \geq 6$); see also \cite{ellis-survey} for details of this construction. A beautiful conjecture of Alon and Spencer, however, says that it should be possible to improve on the trivial upper bound for $P_3$, by an absolute constant factor.
\begin{conj}[Alon, Spencer]
    There exists an absolute constant $c>0$ such that $g_{P_3}(n) \leq (1-c)2^{{n \choose 2}-1
}$ for all $n \geq 4$.
\end{conj}
Part of the beauty of this conjecture is that it would characterize exactly for which fixed unlabelled graphs $H$ it holds that $g_{H}(n) = (1-o(1))2^{{n \choose 2}-1}$, i.e.\ for which $H$ one can construct an $H$-intersecting family of graphs on $[n]$ containing asymptotically half of all graphs on $[n]$. (It is easy to see that this can be done when $H$ is a forest of stars, and any graph $H$ that is not a forest of stars must contain either a path of length three, or a triangle.) 

However, again somewhat embarrassingly, even the following much weaker conjecture remains open.
\begin{conj}
\label{conj:bicliques}
    There exists $t \in \mathbb{N}$ and an absolute constant $c>0$ such that $g_{K_{t,t}}(n) \leq (1-c)2^{{n \choose 2}-1}$ for all $n \in \mathbb{N}$.
\end{conj}
It seems very hard indeed to beat the `trivial' upper bound (by any absolute constant factor) when the fixed graph $H$ is bipartite. Proving Conjecture \ref{conj:bicliques} in the case $t=2$, i.e.\ for the 4-cycle (which of course implies all the other cases) would of course be particularly nice.

At the other extreme, it would be of great interest to characterise the set of unlabelled graphs $H$ such that for any $n \in \mathbb{N}$, $g_H(n) = 2^{{n \choose 2}-e(H)}$, or equivalently, for any $n \in \mathbb{N}$, a maximum-sized $H$-intersecting family of graphs on $[n]$ can be obtained by taking all graphs containing a fixed copy of $H$. (For brevity, we call this the {\em Erd\H{o}s-Ko-Rado property} for $H$.) The results and facts listed above imply that $H$ has the Erd\H{o}s-Ko-Rado property in the case when $H$ is an edge, a triangle or a $K_4$, but not when $H$ is a path of three edges, or a disjoint union of stars. Conjecture \ref{conj:eff} would imply that $K_t$ has the Erd\H{o}s-Ko-Rado property, for all $t \geq 5$. The author conjectured in 2012 that $H$ being 2-connected is sufficient to guarantee the Erd\H{o}s-Ko-Rado property, but this was recently disproved by Balogh and Linz \cite{bl}, who showed that when $H=K_{s,t}$ and $t>2^{2s}-2s-1$, we have $g_{H}(n) \geq (t+2s+1)^s 2^{{n \choose 2}-st-2s^2} > 2^{{n \choose 2}-e(H)}$ for all $n \geq t+3s$, showing that even $s$-connectivity does not suffice (for any $s \in \mathbb{N}$). We conjecture now that, to guarantee the Erd\H{o}s-Ko-Rado property, it suffices for $H$ to be a {\em critical graph} (meaning, all proper subgraphs of $H$ have smaller chromatic number than $H$).

\pagebreak

\section{No three points on a line (Joshua Erde).}

A classic problem of Dudney asks how many points can be placed in an $n \times n$ grid so that no three points lie on a common line. There is an obvious upper bound of $2n$, since each row contains at most two points, however, perhaps surprisingly, in general it is not known if this bound can be improved even by a constant additive factor.

On the other hand, a construction of Erd\H{o}s \cite{E51} shows that we can take $(1-o(1))n$ points --- the idea is to consider a \emph{modular parabola}, the set of points $(i,i^2) \mod n$ for prime $n$. Slightly better constructions are known, but they all have quite a rigid algebraic structure.

In the spirit of the work of Erd\H{o}s \cite{E41} on the density of infinite Sidon sets, and the more recent work of Sudakov, Tomon and Wagner \cite{STW22} on the density of antichains in the infinite Boolean lattice, where the density of optimal finite examples cannot be achieved in the limit, we conjecture that a similar phenomenon occurs in this setting.

More concretely, we make the following.
\begin{conj}
If $S \subseteq \mathbb{N}^2$ is such that no three points of $S$ lie on a common line, then
\[
\liminf_{n \to \infty} \frac{|S \cap [n]^2|}{n} =0.
\]
\end{conj}

The author notes that this is perhaps a `bad' conjecture, in that it may be false. Indeed, recently Nagy, Nagy and Woodroofe \cite{NNW23} gave a construction of a set $S \subseteq \mathbb{N}^2$ such that
\[
\liminf_{n \to \infty} \frac{|S \cap [n]^2| (\log n)^{1+\epsilon}}{n} > 0
\]
for any $\epsilon >0$, and worse still, computational analysis suggests that even a greedy approach may well achieve $|S \cap [n]^2| = (1+o(1)) 0.8n$ as $n \to \infty$. However, it would be interesting to disprove the conjecture (if it is false), and if it can be disproved, finding an explicit construction which disproves it also seems to be an interesting (and difficult) thing to do!

\pagestyle{empty}

\pagebreak

\section{Asymmetry of transition probabilities in the two-coloured torus (Ron Gray and Robert Johnson).}

Bollob\'as and Leader famously proved edge-isoperimetric inequalities in the grid \cite{BL1} and in the discrete torus \cite{BL2}, determining which subsets of given size minimise the edge-boundary\footnote{To be completely precise, the inequality in \cite{BL2} is sharp for all set sizes; the inequality in \cite{BL1} is not sharp for every set size, but it is sharp for many set sizes.}. This problem is about a different style of extremal question also related to measures of boundary of subsets of the torus.

In a regular graph, the edge-boundary is closely related to the number of internal edges. It is also easy to see that in any partition of the vertex set of a regular graph into equal sized parts, the two parts have the same number of internal edges. However, if instead of internal edges we count walks of length two that stay within one part, then this symmetry disappears. How extreme can the asymmetry be?

More precisely, let $\mathbb{T}_n^k$ denote the graph of the $k$-dimensional discrete torus, i.e.\ the graph with vertex-set $\mathbb{Z}_n^k$, where $x$ is joined to $x \pm e_i$ for all $i \in [k]$ and all $x \in \mathbb{Z}_n^k$, with $e_i$ denoting the unit vector in the $i$th direction, as usual. Take a red-blue colouring of the vertices of $\mathbb{T}_n^k$ with an equal number of vertices in each colour. Let $p_R$ be the probability that a random walk starting from a random red vertex stays within the red class for at least two steps, and let $p_B$ be the corresponding probability for blue. For which pairs $(x, y)\in [0,1]^2$ is there such a colouring with $(p_R,p_B)=(x,y)$?

In \cite{GJ} it is shown that for $k=2$ the answer is essentially (once the appropriate large $n$ limit has been defined) the convex hull of the set:
\[
\left\{\left(0,0\right),\left(\frac{1}{2},\frac{1}{4}\right),\left(\frac{3}{4},\frac{9}{16}\right),\left(\frac{1}{4},\frac{1}{2}\right),\left(\frac{9}{16},\frac{3}{4}\right),\left(1,1\right)\right\}
\]

Our main question is: what happens for $k\geq 3$? A tentative conjecture is that the answer is the convex hull of the set: 
\[
\left\{(0,0)\right\}\cup\left\{\left(\frac{l}{2k},\frac{l^2}{(2k)^2}\right):k \leq l \leq 2k\right\}\cup\left\{\left(\frac{l^2}{(2k)^2},\frac{l}{2k}\right):k \leq l \leq 2k\right\}.
\]

Another interesting generalisation is to consider the corresponding pairs of probabilities for $t$-step walks with $t\geq 3$. Here, even the $k=2$ case is open. 

See \cite{GJ} for more details, partial results and background.

\pagestyle{empty}

\pagebreak

\section{Some questions in Ramsey theory, poset saturation and graph colouring (Maria-Romina Ivan).}

\subsection*{Ramsey Theory}

\begin{question}[\cite{wordsMRI}] Suppose the natural numbers are coloured with finitely many colours. Must there exist a
sequence $(y_n)_{n=1}^{\infty}$ such that every sum $a_1 y_1 + a_2 y_2 +
\cdots + a_k y_k$
with $a_i \in \{1,2\}$ for all $i$ and $a_1=a_k=1$, has the same
colour?\end{question}

This question arose from the work done in \cite{wordsMRI}, where it was required to build a special colouring of $\mathbb{N}^2$. We wondered if we could colour $(a,b)$ by first colouring the naturals and giving the colour of $a+b$ to $(a,b)$. This reduced to the above question, which remains unanswered despite the fact that the original question concerning $\mathbb{N}^2$ was solved in \cite{wordsMRI}.

\begin{question}[F. Richter, {\em Ultramath} 2022, Pisa.] Suppose the natural numbers are coloured with finitely many colours. Must there exist two infinite sets of natural numbers, $A$ and $B$, such that the set
$$\{ab:\ a\in A,\ b\in B\} \cup \{a+b:\ a \in A,\ b \in B\}$$
is monochromatic?
\end{question}

This question does not need much introduction. It is extremely natural, in the light of the wonderful results of Moreira \cite{M} that $\{x, x+y, xy\}$ is partition regular over the natural numbers, and of Bowen and Sabok \cite{MB} that $\{x,y,x+y, xy\}$ is partition regular over the rational numbers. The above question is a possible `infinitary' result in the spirit of these.

\subsection*{Poset Saturation}

\begin{question}[Ferrara, Kay, Kramer, Martin, Reiniger, Smith and Sullivan, \cite{fk}] Let $\mathcal{F}$ be a family of subsets of $[n]$ that does not contain a {\em diamond} as an induced poset (meaning, $\mathcal{F}$ does not contain four distinct sets $A$, $B$, $C$ and $D$ with $A \subset B \cap C$, $B \cup C \subset D$, $C \not\subset B$ and $B \not\subset C$). Assume further that $\mathcal{F}$ is maximal with respect to the above condition, i.e.\ that for any $A\notin\mathcal F$, the family $\mathcal F\cup\{A\}$ does contain a diamond as an induced poset. Is it true that $|\mathcal{F}| = \Omega(n)$?
\end{question}

It is easy to see that the set of all singletons, together with the empty set, is a family satisfying the above conditions, and it has size $n+1$. The same holds for the full chain, $\{\emptyset, 1, 12, 123,\ldots,[n]\}$.

This question falls within the area of study known as `poset saturation'. A family $\mathcal{F}$ with the two properties above (diamond-free, and maximal subject to this condition) is said to be {\em diamond-saturated}. All that is known about this question so far is that the size of a diamond-saturated family is at least $\sqrt n$; it is believed that it must be at least linear. It is even open to show that it must be $\omega(\sqrt n)$. See \cite{diamond,martin2019improved} for more details. 

\subsection*{Graph Colouring}

\begin{question} [Bonamy, Botler, Dross, Naia and Skokan, \cite{rainbow}] Suppose that the edges of an $n$-vertex graph are coloured in such a way that no two incident edges have the same colour. Can we always cover the edge-set of the graph with $O(n)$ rainbow paths (meaning, paths that do not use the same colour twice)?
\end{question}

Very little is known about this natural question (see \cite{rainbow}). If one greedily removes the longest rainbow path, one obtains a covering of the edge-set of an $n$-vertex graph with $n\log n$ rainbow paths. Can one do better? This question is particularly intriguing because one cannot work with `small' cases to get a feel for the problem.

\pagebreak

\section{Antipodal paths in $Q_n$ with few colour changes (Maria-Romina Ivan and Robert Johnson).}

\begin{conj} [Leader and Long, \cite{geodesics}] Suppose the edges of the hypercube $Q_n$ are coloured with two colours, red and blue. Then there exists a geodesic path between two antipodal vertices, with at most one colour-change.

Equivalently, if the edges are 2-coloured `antipodally' (meaning that the edges $vw$ and $\bar{v}\bar{w}$ always have different colours, where for a vertex $v$, $\bar{v}$ is vertex disagreeing with $v$ on all coordinates), there exists a monochromatic geodesic path between two antipodal vertices.
\end{conj}

Here, a {\em geodesic path} means a shortest path in $Q_n$, i.e.\ one in which no direction is repeated. Note that a geodesic path between antipodal vertices must have length $n$. The equivalence of the two statements in the above conjecture is an easy exercise.

This conjecture strengthens a conjecture of Norine \cite{norine}, that if the edges of $Q_n$ are antipodally coloured with two colours, then there exists a monochromatic path of some length between two andipodal vertices. It also strengthens a conjecture of Feder and Subi \cite{fs}, that if the edges of $Q_n$ are 2-coloured, there exists a path of some length between two antipodal vertices, with at most one colour-change (this conjecture of Feder and Subi implies Norine's conjecture).

As a step towards this Conjecture Leader and Long \cite{geodesics} pose the following question.

\begin{question} 
\label{q:o(n)}
Is it true that for every red-blue colouring of the edges of $Q_n$, there exists a path between two antipodal vertices that changes colour $o(n)$ times?
\end{question} 

It follows from a result of Leader and Long \cite{geodesics} that there is always a monochromatic geodesic path of length at least $n/2$. One can also build a colouring where any geodesic path that starts in the largest monochromatic component, changes colour at least $\Omega(n)$ times; this makes the question even more intriguing.

This result immediately implies that there is a geodesic path between two antipodal vertices, with at most $n/2$ colour-changes. Dv\u{o}r\'ak \cite{Dvorak} improved this to $\left(\frac{3}{8}+o(1)\right)n$ colour-changes.

One ingredient of Dvorak's argument is careful analysis of geodesics between \emph{all} antipodal pairs of a $Q_3$. Somewhat motivated by this, we suggest the following variant of the problem. This variant may be more amenable to a proof by induction, than the original problem.

\begin{question}
How large can be the minimum number of colour-changes on a path between a randomly chosen pair of antipodal vertices? In other words, determine the maximum (over all edge colourings of $Q_n$) of the quantity:
\[
\frac{1}{2^n}\sum_{v\in Q_n}\left(\text{the minimum number of colour changes on a $v$-to-$\overline{v}$ path}\right)
\]
\end{question}

One could also pose the same question for geodesic paths. Colouring the hypercube in layers shows that the quantity in the above question must be at least $c\sqrt{n}$, where $c>0$ is an absolute constant. If this is tight, then certainly some antipodal pair has a path between them with $c\sqrt{n}$ colour-changes, giving a positive answer to Question \ref{q:o(n)}

Another direction is the extension to colourings of the edges of $Q_n$ with more than two colours. The analogue of Leader and Long's conjecture for more than two colours, would be the following.

\begin{conj}
Let $k \in \mathbb{N}$. In any $k+1$-colouring of the edges of $Q_n$, there exists a geodesic path between two antipodal vertices, with at most $k$ colour-changes.
\end{conj}

Any colouring of the edges of $Q_n$ in which all edges in the same direction have the same colour shows that this would be best-possible. More colours gives more scope for complicated colourings, so perhaps this conjecture is too strong and it makes more sense to try to disprove it. Is there some $k$ and $n$ and a $(k+1)$-colouring of the edges of $Q_n$ in which every antipodal path has strictly more than $k$ colour-changes? For an even more ambitious disproof, could there be some fixed $k$ for which the smallest number of colour-changes in an antipodal path (or geodesic) in a $(k+1)$-coloured $Q_n$ is unbounded as $n$ grows?

\pagebreak

\section{Convex $S$-Kakeya sets (Barnab\'as Janzer).}
	
	A subset $K$ of $\mathbb{R}^d$ is called a \emph{Kakeya set} or \emph{Besicovitch set} if it contains a unit line segment in all possible orientations, i.e., whenever $v\in \mathbb{R}^d$ with $\|v\|_2=1$, there exists $x\in \mathbb{R}^d$ such that for all $\lambda\in [0,1]$ we have $x+\lambda v\in K$. The main foundational results about such sets were proved by Besicovitch~\cite{besicovitch1,besicovitch1928kakeya}, who answered a question originating from Kakeya by showing that, surprisingly, there exist sets in $\mathbb{R}^2$ of arbitrarily small measure in which a unit line segment can be continuously moved and rotated around by $360^\circ$, and also showed that there exist Kakeya sets of measure zero in the plane. Since then there has been a lot of interest in Kakeya sets and related problems in various areas of mathematics. Here we consider a geometric problem of this kind raised by H.~Croft.
	
	Generalising the notion of Kakeya sets, given non-empty sets $S$ and $K$ in $\mathbb{R}^d$, let us say that $K$ is \emph{$S$-Kakeya} if $K$ contains a copy of $S$ in all possible orientations, i.e., if whenever $\rho\in \operatorname{SO}(d)$ is a rotation, there exists $x\in \mathbb{R}^d$ such that $\rho(S)+x\subseteq K$. Note that when $S$ is a unit line segment, we get the usual notion of Kakeya sets. Croft asked the following question: if $S$ and $K$ are convex bodies in $\mathbb{R}^d$, and $K$ is $S$-Kakeya, does it follow that any $S$-copy in $K$ can be continuously rotated into any other $S$-copy in $K$? (This question can be formalised by saying that \emph{any two $S$-copies can be rotated into each other within $K$} if whenever $\rho_0,\rho_1\in \operatorname{SO}(d)$ are rotations and $x_0,x_1\in \mathbb{R}^d$ are such that $\rho_i(S)+x_i\subseteq K$ (for $i=0,1$), there exist continuous functions $\gamma:[0,1]\to \operatorname{SO}(d)$ and $\delta: [0,1]\to \mathbb{R}^d$ such that $\gamma(i)=\rho_i,\ \delta(i)=x_i$ for $i=0,1$, and $\gamma(t)(S)+\delta(t)\subseteq K$ for all $t\in[0,1]$. By a {\em body}, we mean a non-empty compact set.)
	
	In~\cite{janzer2022rotation}, it was shown that the answer to Croft's question is negative in general. However, for some very natural choices of $S$ (such as unit line segments) we do have that any two $S$-copies can be rotated into each other inside convex $S$-Kakeya bodies. For brevity, given $S\subseteq\mathbb{R}^d$, let us say that $S$ has the \emph{Kakeya rotation property} if whenever $K$ is a convex $S$-Kakeya body in $\mathbb{R}^d$, any two $S$-copies can be rotated into each other within $K$. The results of~\cite{janzer2022rotation} can now be summarised as follows.
	
	\begin{theorem}[\cite{janzer2022rotation}]\label{kakeyaknown}
		Let $S$ be a convex body in $\mathbb{R}^d$. If $d=2$ (and $S$ is arbitrary), or if $S$ is a unit line segment (and $d\geq 2$ is arbitrary), then $S$ has the Kakeya rotation property.
		
		However, for each $d\geq 4$, there exists a convex body $S$ in $\mathbb{R}^d$ which does not have the Kakeya rotation property.
	\end{theorem}

Theorem~\ref{kakeyaknown} leads to several interesting and natural questions, some of these are also mentioned in~\cite{janzer2022rotation}. For example, the counterexample in Theorem~\ref{kakeyaknown} requires $d\geq 4$, but we know that there exists no $2$-dimensional counterexample. It would be interesting to determine whether there is a counterexample in $3$ dimensions (we believe that such does indeed exist).

\begin{question}
Is there a convex body $S$ in $\mathbb{R}^3$ which does not have the Kakeya rotation property?
\end{question}

Furthermore, we see that some convex sets (such as unit line segments) have the Kakeya rotation property, whereas others do not. Which of these two behaviours is `more common'? For example, what happens for two-dimensional balls (=circular disks) embedded in $\mathbb{R}^d$ (for $d\geq 3$), or for (solid) cubes, or (solid) hypercubes? Can we classify the convex sets $S$ having the Kakeya rotation property, or at least obtain a partial classification in some natural cases? As a start, it would be interesting to find any other convex sets $S$ (or families of such) which do have this property. (A trivial example is given by closed balls.)

\begin{question}
	Are there convex bodies $S$ in $\mathbb{R}^d$ with $d\geq 3$ such that $S$ is not a line segment or a $d$-dimensional ball, and $S$ has the Kakeya rotation property?
\end{question}

There are many other natural questions that arise.

\pagebreak

\section{Pinto's Alternative Compressions: a technique in search of more applications (Robert Johnson).}

The compression method has many applications in extremal set theory (and indeed in other areas of mathematics). The content of this contribution, is to describe in general terms a simple but novel twist on the compression method which was invented by Trevor Pinto \cite{TP} to prove a conjecture of Bollob\'as and Leader for which the natural compression approach looks doomed to fail. The `open problem' here is to find more applications of this idea. 

The most appealing special case of Pinto's result (which answers a conjecture in \cite{BL}) is the following.

\begin{theorem}[Pinto \cite{TP}]
Let $\mathcal{A}\subseteq Q_n$ be a down-set and $\mathcal{B}\subseteq Q_n$ be an up-set with $|\mathcal{A}|=|\mathcal{B}|=2^k$ and $\mathcal{A}$ and $\mathcal{B}$ disjoint. Then there are at least $2^k(n-k)$ edge-disjoint upwards directed paths from $\mathcal{A}$ to $\mathcal{B}$.
\end{theorem} 

A key lemma in proving this was a `directed isoperimetric inequality' which, roughly speaking, establishes that among all families of sets which are sandwiched between two given nested up-sets, the minimum of $\overrightarrow\partial(\mathcal{S})$ (the upwards directed edge boundary) is attained by a down-set.

One way to make a family $\mathcal{F}$ into a down-set in the $i$-direction is to remove from it all sets $S$ which contain $i$ and for which $S\setminus\{i\}\not\in\mathcal{F}$. Another way to do this is to add to it all sets $S$ which don't contain $i$ but for which $S\cup\{i\}\in\mathcal{F}$. This gives us two candidate compression operators:
\begin{align*}
\mathcal{C}_i(\mathcal{F})&=\mathcal{F}\setminus\{ S\in\mathcal{F} : i\in S, S\setminus\{i\}\not\in\mathcal{F}\}\\ 
\mathcal{D}_i(\mathcal{F})&=\mathcal{F}\cup\{ S : i\not\in S, S\cup\{i\}\in\mathcal{F}\}
\end{align*}

To prove the key lemma via one of these compressions we would like to have that our compression cannot increase the directed edge-boundary. Unfortunately it is possible to find a family $\mathcal{F}$ for which $|\overrightarrow\partial(\mathcal{C}_i(\mathcal{F}))|>|\overrightarrow\partial(\mathcal{F})|$ and a family $\mathcal{G}$ with $|\overrightarrow\partial(\mathcal{D}_i(\mathcal{G}))|>|\overrightarrow\partial(\mathcal{G})|$. This appears to prevent using either of these compressions as the basis for a proof.

Trevor's brilliant observation was that it is impossible for a family to be simultaneously bad for both the $\mathcal{C}_i$ compression and the $\mathcal{D}_i$ compression. Indeed, it can be shown that
\[
|\overrightarrow\partial(\mathcal{S})|\geq \frac{1}{2}\left(|\overrightarrow\partial(\mathcal{C}_i(\mathcal{S}))|+|\overrightarrow\partial(\mathcal{D}_i(\mathcal{S}))|\right)\geq \min\left\{|\overrightarrow\partial(\mathcal{C}_i(\mathcal{S}))|,|\overrightarrow\partial(\mathcal{D}_i(\mathcal{S}))|\right\}
\]

This is enough to drive a compression argument, but crucially, one in which the operation used at each step is not fixed; we know that one of the $\mathcal{C}_i$ and $\mathcal{D}_i$ compressions works, but which one depends on the family.

\begin{question}
Find some more applications where this twist on the compression method can be used to rescue a compression proof. That is, arguments in which any single compression operator is not guaranteed to behave well in all situations, but where this bad behaviour cannot occur simultaneously for a larger collection of compressions.
\end{question}

\pagebreak

\section{A modular version of Rado's boundedness conjecture (Luka Mili\'cevi\'c).}

One of the many topics Imre has contributed to is that of partition regularity. A key notion in partition regularity is that of a $k$-regular system. For $k \in \mathbb{N}$, we say an $m \times n$ matrix $A$ with integer entries is \emph{$k$-regular} if for any $k$-colouring of $\mathbb{N}$, there exists a monochromatic $n$-tuple $x = (x_1, \dots, x_n) \in \mathbb{N}^n$ such that $A x = 0$. An important problem in this area is the well-known {\em Rado's boundedness conjecture}, stating that for any positive integers $m$ and $n$ there exists another integer $K$ with the following property: if an $m \times n$ integer matrix $A$ is $K$-regular, then it is $k$-regular for all $k$. (Rado \cite{rado} proved that the $m=1$ case of this conjecture in fact implies the general case.) The only non-trivial case when this conjecture is known is $(m,n) = (1,3)$, which is a theorem of Fox and Kleitman~\cite{FK}. One of the steps in their proof is a pass from integers to $\mathbb{Z}/(q\mathbb{Z})$ for a suitable $q$, where they consider the same problem with one variable removed, which can be solved directly. It is therefore of interest to study the modular setting. In fact, Fox and Kleitman formulate a modular variant of Rado's conjecture for a single equation (Conjecture 5 in \cite{FK}), where they restrict their attention to units in $\mathbb{Z}/(q\mathbb{Z})$. We make the following conjecture, which has no such restriction.

\begin{conj}For any $k \in \mathbb{N}$, there exists $K \in \mathbb{N}$ such that following holds. Suppose that $a_1, \dots, a_k \in \mathbb{Z}/(2^r\mathbb{Z})$, and let $d \leq r$ be the largest non-negative integer such that $2^d$ divides $\sum_{i \in I} a_i$ for some non-empty set $I \subseteq [k]$. Then there exists a $K$-colouring of $\mathbb{Z}/(2^r\mathbb{Z})$ such that all monochromatic solutions $x = (x_1, \dots, x_k)$ to the equation $a_1x_1 + \dots + a_kx_k = 0$ satisfy $2^{r-d} | x_i$ for all $i \in [k]$.
\end{conj}

A standard compactness argument shows that this conjecture implies Rado's boundedness conjecture for a single equation, and therefore (by Rado's above-mentioned result) implies every case of Rado's boundedness conjecture.

\pagestyle{empty}

\pagebreak

\section{A boundedness conjecture for Ramsey sets (Luka Mili\'cevi\'c).}

Another area that Imre has studied is Euclidean Ramsey theory. We say that a finite set $X$ in some Euclidean space $\mathbb{R}^n$ is \emph{Ramsey} if for any $k \in \mathbb{N}$ there exists $d \in \mathbb{N}$ such that whenever $\mathbb{R}^d$ is $k$-coloured, there exists a monochromatic congruent copy of $X$.

A central question in Euclidean Ramsey theory is that of characterizing Ramsey sets. Erd\H{o}s, Graham, Montgomery, Rothschild, Spencer and Straus ~\cite{Ramsey} proved that a Ramsey set must be {\em spherical} (meaning, all points thereof lie on some sphere). A well-known question of Erd\H{o}s, Graham, Montgomery, Rothschild, Spencer and Straus~\cite{Ramsey} asks whether the converse holds, i.e., are all spherical sets Ramsey? Graham \cite{graham} later conjectured a positive answer to this question.

Interestingly, Imre, together with Russell and Walters~\cite{IRW}, formulated a competing conjecture, stating that the Ramsey sets are precisely the subsets of \emph{transitive sets} (where, as usual, we say a set is {\em transitive} if its symmetry group acts transitively).

Here, in the spirit of Rado's boundedness conjecture, we formulate another conjecture on Ramsey sets.

\begin{conj}
For any $k \in \mathbb{N}$ and $n\in \mathbb{N}$, there exists $d = d(k,n) \in \mathbb{N}$ such that, whenever $X$ is a Ramsey set of size at most $n$ in some Euclidean space, and $\mathbb{R}^d$ is $k$-coloured, there exists a monochromatic congruent copy of $X$.
\end{conj}

\pagestyle{empty}

\pagebreak

\section{Covering the product of complete graphs (Ta Sheng Tan).}

The Graham-Pollak problem asks for a determination of $f_r(n)$, defined as the least number of complete $r$-partite $r$-graphs needed to partition the edge set of the complete $r$-graph on $n$ vertices.
The key to the approach used in \cite{leader} and \cite{leader2} in proving the (currently best-known) general upper bounds for $f_r(n)$ was to investigate $g(n)$, defined as the minimum number of products of complete bipartite graphs (that is, sets of the form $E(K_{a,b})\times E(K_{c,d})$) needed to partition the set $E(K_n)\times E(K_n)$.
It is trivial that $g(n)\le (n-1)^2$, by taking the products of the complete bipartite graphs appearing in a decomposition of $K_n$ into $n-1$ stars.
In \cite{leader}, it was shown that
$$\left(\frac{1}{2}+o(1)\right)n^2 \le g(n)\le \left(\frac{14}{15}+o(1)\right)n^2,$$
but the problem of determining the asymptotic of $g(n)$ remains open.
The above upper bound on $g(n)$ was obtained by investigating the minimum number $h(n)$ of parts needed in a partition of $E(K_4)\times E(K_n)$ into products of complete bipartite graphs. While the asymptotic of $h(n)$ is not known, we conjecture in \cite{leader} that $h(n) = \left(\frac{12}{5}+o(1)\right)n$. We ask the following.

\begin{question}[\cite{leader}]
 Is it true that $g(n)=\left(\frac{4}{5}+o(1)\right)n^2$?
\end{question} 

The upper bound $g(n) \leq \left(\frac{4}{5}+o(1)\right)n^2$ would follow from the upper bound $h(n) \leq \left(\frac{12}{5}+o(1)\right)n$, as a partition of $E(K_n) \times E(K_n)$ can be built out of a partition of $E(K_4) \times E(K_n)$ in a natural way. (See \cite{leader} for more details on the motivation.)

What if we instead allow each element of $E(K_n)\times E(K_n)$ to be in an odd number of products of complete bipartite graphs?
Let $\tilde{g}(n)$ be the minimum number of products of complete bipartite graphs needed so that each element of $E(K_n)\times E(K_n)$ is in an odd number of these products of complete bipartite graphs.

\begin{question}
 What are the values of $\tilde{g}(n)$?
\end{question}

Of course, we have $\tilde{g}(n)\le g(n)$.
But in fact, using a result of \cite{buchanan} about the smallest odd cover of $K_n$, it is straightforward that $\tilde{g}(n)\le \left(\left\lceil\frac{n}{2}\right\rceil+1\right)^2$.
Similarly to how the function $g$ is related to $f_r$, one can use $\tilde{g}$ to study the odd cover variant of the Graham-Pollak problem.
For a fixed $r$, let $\tilde{f}_r(n)$ be the minimum number of complete $r$-partite $r$-graphs needed so that each edge of $K_n^{(r)}$ is in an odd number of these complete $r$-partite $r$-graphs.
Using almost identical arguments as in \cite{leader} and \cite{leader2}, one can obtain general upper bounds of $\tilde{f}_r(n)$ based on $\tilde{g}(n)$.
In particular, we have $\tilde{f}_4(n)\le \left(\frac{1}{4}+o(1)\right)\binom{n}{2}$, answering Problem 2 of \cite{buchanan} affirmatively for the case when $r=4$.
Finally, we remark that the lower bound of $\tilde{f}_4(n)$ mentioned in \cite{buchanan} gives a lower bound of $\tilde{g}(n)\ge \left(\frac{1}{6}+o(1)\right)n^2$.

\pagebreak

\section{Partially shattering $k$-sets with permutations (Belinda Wickes).}

Many concepts from extremal set theory have interesting analogues for families of
permutations. This problem is concerned with the notion of shattering. A family $\mathcal{P}$
of permutations from $S_n$ {\em shatters} a $k$-set $X \subset [n]$ if the elements of $X$ appear in each of the
$k!$ possible orders in permutations from $\mathcal{P}$.
For example, the following family of permutations from $S_5$ shatters the triple
$\{2, 3, 5\}$, but not $\{1, 4, 5\}$, since no permutation gives the order $(4, 5, 1)$.
$$\{(1, 2, 3, 4, 5), (2, 4, 1, 5, 3), (5, 3, 4, 1, 2),(1, 4, 3, 5, 2), (3, 1, 2, 5, 4), (5, 1, 2, 4, 3)\}$$

The smallest family $\mathcal{P}$ which shatters every $k$-subset of $[n]$ is known to have size
$\Theta(\log n)$ (\cite{rad} and \cite{spen}).

Here, we consider a natural partial version of this shattering
problem. A family $\mathcal{P} \subseteq S_n$ is said to {\em $t$-shatter} a set $X \subseteq [n]$ if the elements of $X$ appear
in at least $t$ distinct orders among the permutations of $\mathcal{P}$. We will be interested in
small families of permutations which $t$-shatter every $k$-subset of $[n]$. For example, the
family above 4-shatters every 3-subset of $[5]$. It is proved in \cite{wickes} that the minimum size of a family $\mathcal{P}$ which $t$-shatters every
3-subset of $[n]$ falls into three distinct regimes.
\begin{theorem}
$$\min\{|\mathcal{P}|:\ \mathcal{P}\ t\text{-shatters every 3-subset of }[n]\} = \begin{cases} t & \text{if } t=1,2;\\ \Theta(\log \log n) & \text{if } t=3,4; \\ \Theta(\log n) & \text{if } t=5,6.\end{cases}$$
\end{theorem}
The main problem is to extend this result to $k > 3$. The following partial result
is proved in \cite{wickes}.

\begin{theorem}
$$\min\{|\mathcal{P}|:\ \mathcal{P}\ t\text{-shatters every }k\text{-subset of }[n]\} = \begin{cases} t & \text{if } t = 1,2;\\ 
\Theta(\log \log n) & \text{if } t \in [3, k];\\
\Theta(\log n) & \text{if } t \in [2(k-1)!+1,k!].\end{cases}$$
\end{theorem}

As before, the same three regimes appear, but it is unknown what happens to the
size of $\mathcal{P}$ when $t$ lies in the range $[k+1,2(k-1)!]$. In particular, we have the following
questions:
\begin{itemize} \item Are there any values of $t$ for which the size of an extremal family is not in one
of these three regimes?
\item If not, then for which $t$ does the transition from $\Theta(\log \log n)$ to $\Theta(\log n)$ occur?
\end{itemize}

\begin{thebibliography}{99}
\bibitem{B} R. Baber. Tur\'an densities of hypercubes. Preprint. arXiv:1201.3587 (2012).
\bibitem{EC} P. Erd\H{o}s. Some problems in graph theory, combinatorial analysis and combinatorial number theory. In: {\em Graph Theory and Combinatorics: Proceedings of the Cambridge Combinatorial Conference, in Honour of Paul Erd\H{o}s}, B. Bollob\'as, Ed., Academic Press, Cambridge MA, 1984. pp.\ 1--17.
\bibitem{K} E. A. Kostochka. Piercing the edges of the $n$-dimensional unit cube. {\em Metody Diskretnogo Analiza v Teorii Grafov i Logiceskih Funkcii, Diskret. Analiz Vyp.} \textbf{28}, 55--64 (1976). [In Russian.]
\end{thebibliography}

\begin{thebibliography}{99}
	\bibitem{Behague} N. Behague. Semi-perfect 1-factorizations of the hypercube. {\em Discr. Math.} \textbf{342}, 1696--1702 (2019). 
	
	\bibitem{Bryant} D. Bryant, B. M. Maenhaut and I. M. Wanless. A family of perfect factorisations of complete bipartite graphs. {\em J. Combin. Theory, Series A} \textbf{98}, 328--342 (2002).

    \bibitem{Kotzig} A. Kotzig and J. Labelle. Strongly Hamiltonian graphs. {\em Utilitas Math.}, \textbf{14}, 99--116, (1978).
	\bibitem{Laufer}
	P.J. Laufer. On strongly {H}amiltonian complete bipartite graphs. {\em Ars Combin.} \textbf{9}, 43--46 (1980).
	
	
\end{thebibliography}

\begin{thebibliography}{99}

\bibitem{calbet} A. Calbet. $K_r$-saturated Graphs and the Two Families Theorem. Preprint. arXiv:2302.13389 (2023).

\bibitem{scott} A. Scott and E. Wilmer. Combinatorics in the exterior algebra and the Bollob\'as two families theorem. {\em J. Lond. Math. Soc.} \textbf{104}, 1812--1839 (2021).
\end{thebibliography}

\begin{thebibliography}{99}
\bibitem{bl} J. Balogh and W. Linz. Short proofs of three results about intersecting systems. Preprint. arXiv:2104.00778 (2021).
\bibitem{bz} A. Berger and Y. Zhao. $K_4$-intersecting families of graphs. {\em J. Combin. Theory, Series B} \textbf{163}, 112--132 (2023).
\bibitem{christofides} D. Christofides. Denser families of graphs intersecting in paths of length three. Manuscript, 2011.
\bibitem{cfgs} F. K. R. Chung, R. L. Graham, P. Frankl and J. B. S. Shearer. Some intersection
theorems for ordered sets and graphs. {\em J. Comb. Theory, Series A} \textbf{43}, 23-–37 (1986).
\bibitem{ellis-survey} D. Ellis. Intersection Problems in Extremal Combinatorics: Theorems, Techniques and Questions Old and New. In: A. Nixon and S. Prendiville (Eds.), {\em Surveys in Combinatorics} 2022 (London Mathematical Society Lecture Note Series, pp.\ 115-173). Cambridge University Press, Cambridge, 2022.
\bibitem{eff} D. Ellis, Y. Filmus and E. Friedgut. Triangle-intersecting families of graphs. {\em J. Eur. Math. Soc.} \textbf{14}, 841--885 (2012).
\end{thebibliography}

\begin{thebibliography}{99}
\bibitem{E41} P. Erd\H{o}s and P. Tur\'an. On a problem of Sidon in additive number theory, and on some related problems. {\em J. Lond. Math. Soc.} \textbf{16}, 212--215 (1941).

\bibitem{E51} K. F. Roth. On a problem of Heilbronn. {\em J. Lond. Math. Soc.} \textbf{1}, 198--204 (1951).

\bibitem{STW22} B. Sudakov, I. Tomon and A. Z. Wagner. Infinite Sperner's theorem. {\em J. Combin. Theory, Series A}
\textbf{187}, 105--558 (2022).
  
\bibitem{NNW23} D. T. Nagy, Z. L Nagy and R. Woodroofe. The extensible No-Three-In-Line problem. {\em Europ. J. Combin.} \textbf{114}, 103--796 (2023).

\end{thebibliography}

\begin{thebibliography}{99}
\bibitem{BL1} B. Bollob\'as and I. Leader. Edge-isoperimetric inequalities in the grid. \emph{Combinatorica} \textbf{11}, 299--314 (1991).
\bibitem{BL2} B. Bollob\'as and I. Leader. An isoperimetric inequality on the discrete torus. \emph{SIAM J. Disc. Math.} \textbf{3} 32--37 (1990).
\bibitem{GJ} R. Gray and J. R. Johnson. Asymmetry of $2$-step transit probabilities in $2$-coloured regular graphs. Preprint. arXiv:2307.06054 (2023). 
\end{thebibliography}

\begin{thebibliography}{99}
\bibitem{rainbow} M. Bonamy, F. Botler, F. Dross, T. Naia and J. Skokan. Separating the edges of a graph by a linear number of paths. {\em Adv. Combin.} 2023:6, 7 pp.

\bibitem{MB} M. Bowen and M. Sabok. Monochromatic products and sums in the rationals. Preprint. arXiv:2210.12290 (2022).

\bibitem{fk} M. Ferrara, B. Kay, L. Kramer, R. R. Martin, B. Reiniger, H. C. Smith and E. T. Sullivan. The
saturation number of induced subposets of the Boolean lattice. {\em Discrete Math.} \textbf{340}, 2479--2487 (2017).


\bibitem{diamond} M.-R. Ivan. Minimal diamond-saturated families. {\em Contemp. Math.} \textbf{3}, 81--88 (2022).


\bibitem{wordsMRI} M.-R. Ivan, I. Leader and L. Q. Zamboni. A Ramsey characterisation of eventually periodic words. {\em Bull. Lond. Math. Soc.} \textbf{54}, 2437--2455 (2022).


\bibitem{martin2019improved} R. R. Martin, H. C. Smith and S. Walker. Improved bounds for induced poset saturation. {\em Electronic J. Combin.} \textbf{27}, P2.31 (2020).

\bibitem{M} J. Moreira. Monochromatic sums and products in $\mathbb{N}$. {\em Ann. Math.} \textbf{185}, 1069--1090 (2017).

\end{thebibliography}

\begin{thebibliography}{99}
\bibitem{Dvorak} V. Dvo\u{r}\'ak. A Note on Norine's Antipodal-Colouring Conjecture. {\em Electronic J. Combin.} \textbf{27}, P2.26 (2020).

\bibitem{fs} T. Feder, C. Subi. On hypercube labellings and antipodal monochromatic paths. {\em Discrete Appl. Math.} \textbf{161}, 1421--1426 (2013).

\bibitem{geodesics} I. Leader and E. Long. Long geodesics in subgraphs of the hypercube. {\em Discrete Math.} \textbf{326}, 29--33 (2014).

\bibitem{norine} S. Norine. Open Problem Garden (2010).\\
\texttt{http://garden.irmacs.sfu.ca/?q=op/edge antipodal colorings of cubes.}

\end{thebibliography}

\begin{thebibliography}{99}
\bibitem{janzer2022rotation} B. Janzer. Rotation inside convex Kakeya sets. Preprint. arXiv:2209.09728 (2022).
 
\bibitem{besicovitch1} A. S. Besicovitch. Sur deux questions de l'integrabilit\'e. {\em Journal of the Physical Mathematical Society of the Perm State University}
\textbf{2}, 105--123 (1919).

\bibitem{besicovitch1928kakeya} A. S. Besicovitch. On Kakeya's problem and a similar one. {\em Math. Z.} \textbf{27}, 312--320 (1928).
\end{thebibliography}

\begin{thebibliography}{99}
\bibitem{BL} B. Bollob\'as and I. Leader. Matchings and paths in the cube. \emph{Discrete Appl. Math.} \textbf{75}, 1--8 (1997).
\bibitem{TP} T. Pinto. The proofs of two directed paths conjectures of Bollob\'as and Leader. \emph{Combin. Probab. Comput.} \textbf{26}, 762--774 (2017).

\end{thebibliography}

\begin{thebibliography}{99}

\bibitem{FK} J. Fox and D.J. Kleitman. On Rado's boundedness conjecture. {\em J. Combin. Theory, Series A}  \textbf{113}, 84--100 (2006).

\bibitem{rado} R. Rado. Studien zur Kombinatorik. {\em Math. Z.} \textbf{36}, 242--280 (1933).

\end{thebibliography}

\begin{thebibliography}{99}

\bibitem{Ramsey} P. Erd\H{o}s, R. L. Graham, P. Montgomery, B. L. Rothschild, J. Spencer and E. G. Straus. Euclidean Ramsey Theorems, I. {\em J. Combin. Theory, Series A} \textbf{14}, 341--363 (1973).

\bibitem{graham} R. L. Graham. Recent trends in Euclidean Ramsey theory. {\em Discrete
Math.} \textbf{136}, 119--127 (1994).

\bibitem{IRW} I. Leader, P. A. Russell and M. Walters. Transitive sets in Euclidean Ramsey theory. {\em J. Combin. Theory, Series A} \textbf{119}, 382--396 (2012).

\end{thebibliography}

\begin{thebibliography}{99}
 \bibitem{buchanan} C. Buchanan, A. Clifton, E. Culver, J. Nie, J. O’Neill, P. Rombach and M. Yin. Odd covers of graphs. {\em J. Graph Theory} \textbf{104}, 420--439 (2023).
 \bibitem{leader} I. Leader, L. Mili\'{c}evi\'{c} and T. S. Tan. Decomposing the complete $r$-graph. {\em J. Combin. Theory, Series A} \textbf{154}, 21--31 (2018).

\bibitem{leader2} I. Leader and T. S. Tan. Improved bounds for the Graham-Pollak problem for hypergraphs. {Electron. J. Combin.} \textbf{25}, \#P1.4 (2018).
\end{thebibliography}

\begin{thebibliography}{99}
\bibitem{wickes} J. R. Johnson and B. Wickes. Shattering $k$-sets with permutations. {\em Order} (2023). \texttt{https://doi.org/10.1007/s11083-023-09637-7}.

\bibitem{rad} J. Radhakrishnan. A note on scrambling permutations. {\em Rand. Struct. Algor.} \textbf{22}, 435--439 (2003).

\bibitem{spen} J. Spencer. Minimal scrambling sets of simple orders. {\em Acta Math. Acad. Sci. Hungar.} \textbf{22}, 349--353 (1971).
\end{thebibliography}
\end{document}